\documentclass{amsart}
\usepackage{latexsym,rotate,eucal,cite}
\usepackage{amsmath,amsthm,amssymb,amsxtra}
\usepackage{cite}
\newtheorem{theorem}{Theorem}[section]
\newtheorem{lemma}[theorem]{Lemma}

\theoremstyle{definition}

\theoremstyle{remark}
\newtheorem{remark}[theorem]{Remark}

\numberwithin{equation}{section}

\begin{document}

\title{On biharmonic hypersurfaces with constant scalar curvatures in $\mathbb E^5(c)$}

\author{Yu Fu}
\address{School of Mathematics and
Quantitative Economics, Dongbei University of Finance and Economics,
Dalian 116025, P. R. China}
\email{yufudufe@gmail.com}
\thanks{The author was supported by the
Mathematical Tianyuan Youth Fund of China (No. 11326068), Project
funded by China Postdoctoral Science Foundation (No. 2014M560216).}



\date{february 21, 2014 and, in revised form, September 18, 2014.}



\begin{abstract}
We prove that proper biharmonic hypersurfaces with constant scalar
curvature in Euclidean sphere $\mathbb S^5$ must have constant mean
curvature. Moreover, we also show that there exist no proper
biharmonic hypersurfaces with constant scalar curvature in Euclidean
space $\mathbb E^5$ or hyperbolic space $\mathbb H^5$, which give
affirmative partial answers to Chen's conjecture and Generalized
Chen's conjecture.
\end{abstract}

\maketitle
\section{Introduction}
\hspace*{\parindent}
Biharmonic maps
$\phi:(M^n,g)\longrightarrow(\bar{M}^m,\langle,\rangle)$ between
Riemannian manifolds are critical points of the bienergy functional
\begin{eqnarray*}
E_2(\phi)=\frac{1}{2}\int_M|\tau(\phi)|^2v_g,
\end{eqnarray*}
where $\tau(\phi)= {\rm trace \nabla d\phi}$ is the tension field of
$\phi$ that vanishes for harmonic maps. The Euler-Lagrange equation
associated to the bienergy, which characterizes biharmonic maps, is
given by the vanishing of the bitension field
\begin{eqnarray*}
\tau_2(\phi)=-\Delta\tau(\phi)-{\rm trace}\,
R^{\bar{M}}(d\phi,\tau(\phi))d\phi=0,
\end{eqnarray*}
where $R^{\bar{M}}$ is the curvature tensor of $\bar{M}^m$. The
above equation shows that $\phi$ is a biharmonic map if and only if
its bitension field $\tau_2(\phi)$ vanishes. Equivalently, for an
immersion $\phi:(M^n,g)\longrightarrow(\bar{M}^m,\langle,\rangle)$
between Riemannian manifolds, the mean curvature vector field
$\overrightarrow{H}$ satisfies the following fourth order elliptic
semi-linear PDE
\begin{eqnarray}
\Delta\overrightarrow{H}+{\rm trace}\,
R^{\bar{M}}(d\phi,\overrightarrow{H})d\phi=0.
\end{eqnarray}
In view of (1.1), any minimal immersion, i.e. immersion satisfying
$\overrightarrow{H}=0$, is biharmonic. The non-harmonic biharmonic
immersions are called proper biharmonic.

In a different setting, B. Y. Chen in the middle of 1980s initiated
the study of biharmonic submanifolds in a Euclidean space by the
condition $\Delta \overrightarrow{H}=0$, where $\Delta$ is the rough
Laplacian operator of submanifolds with respect to the induced
metric. Both notions of biharmonic submanifolds in Euclidean spaces
coincide with each other.

Nowadays, the study of biharmonic submanifolds is becoming a very
active subject. There is a challenging biharmonic conjecture of B.
Y. Chen made in 1991 \cite{Chen1991}:

{\bf Chen's conjecture}: {\em The only biharmonic submanifolds of
Euclidean spaces are the minimal ones}.

Due to some non-existence results, Caddeo, Montaldo and Oniciuc
\cite{CMO2001} made in 2001 the following generalized Chen's
conjecture:

{\bf Generalized Chen's conjecture}: {\em Every biharmonic
submanifold of a Riemannian manifold with non-positive sectional
curvature is minimal}.

Up to now, Chen's conjecture is still open. Recently, Generalized
Chen's conjecture was proved to be wrong by Y. L. Ou and L. Tang in
\cite{Ou2012}, who constructed examples of proper-biharmonic
hypersurfaces in a 5-dimensional space of non-constant negative
sectional curvature. However, Generalized Chen's conjecture is still
open in its full generality for ambient spaces with constant
sectional curvature. For more recent developments of Chen's
conjecture and Generalized Chen's conjecture, for instance, see
[1-3, 10-18].

In contrast, the class of proper biharmonic submanifolds in
Euclidean spheres is rather rich and quite interesting. The complete
classifications of biharmonic hypersurfaces in $\mathbb S^3$ and
$\mathbb S^4$ were obtained in \cite{CMO2001, BMO20102}. Moreover,
the authors in \cite{BMO2008} classified biharmonic hypersurfaces
with at most two distinct principal curvatures in $\mathbb S^n$ with
arbitrary dimension. Very recently, biharmonic hypersurfaces with
three distinct principal curvatures in $\mathbb S^n$ were classified
by the author in \cite{fuyu}.

In the present paper, we prove that a biharmonic hypersurface with
constant scalar curvature in the space forms $\mathbb E^{5}(c)$
necessarily has constant mean curvature. As an application of this
result, we show that biharmonic hypersurfaces with constant scalar
curvature in Euclidean space $\mathbb E^5$ and hyperbolic space
$\mathbb H^5$ have to be minimal. Hence, these results give
affirmative partial answers to Chen's conjecture and Generalized
Chen's conjecture.

\section{Preliminaries}
Let $x: M^n\rightarrow\mathbb{E}^{n+1}(c)$ be an isometric immersion
of a hypersurface $M^n$ into a space form $\mathbb{E}^{n+1}(c)$ with
constant sectional curvature $c$. Denote the Levi-Civita connections
of $M^n$ and $\mathbb{E}^{n+1}(c)$ by $\nabla$ and $\tilde\nabla$,
respectively. Let $X$ and $Y$ denote vector fields tangent to $M^n$
and let $\xi$ be a unit normal vector field. Then the Gauss and
Weingarten formulas (cf. \cite{chenbook2011, chenbook1984}) are
given, respectively, by
\begin{eqnarray}
\tilde\nabla_XY&=&\nabla_XY+h(X,Y),\label{l23}\\
\tilde\nabla_X\xi&=&-AX,\label{l16}
\end{eqnarray}
where $h$ is the second fundamental form, and $A$ is the Weingarten
operator. It is well known that the second fundamental form $h$ and
the Weingarten operator $A$ are related by
\begin{eqnarray}\label{l3}
\langle h(X,Y),\xi\rangle=\langle AX,Y\rangle.
\end{eqnarray}
The mean curvature vector field $\overrightarrow{H}$ is given by
\begin{eqnarray}
\overrightarrow{H}=\frac{1}{n}{\rm trace}~h.
\end{eqnarray}
Moreover, the Gauss and Codazzi equations are given, respectively,
by
\begin{eqnarray*}
R(X,Y)Z=c(\langle Y, Z\rangle X-\langle X,Z\rangle Y)+\langle
AY,Z\rangle AX-\langle AX,Z\rangle AY,
\end{eqnarray*}
\begin{eqnarray*}
(\nabla_{X} A)Y=(\nabla_{Y} A)X,
\end{eqnarray*}
where $R$ is the curvature tensor of the hypersurface $M^n$ and
$(\nabla_XA)Y$ is defined by
\begin{eqnarray}\label{l7}
(\nabla_XA)Y=\nabla_X(AY)-A(\nabla_XY)
\end{eqnarray}
for all $X, Y, Z$ tangent to $M^n$.

Assume that $\overrightarrow{H}=H\xi$ and $H$ denotes the mean
curvature.

By identifying the tangent and the normal parts of the biharmonic
condition (1.1) for hypersurfaces in a space form $\mathbb
E^{n+1}(c)$, we obtain the following characterization result for
$M^n$ to be biharmonic (see also \cite{CMO2002, BMO20102}).
\begin{theorem}
The immersion $x: M^n\rightarrow\mathbb{E}^{n+1}(c)$ of a
hypersurface $M^n$ in an $n+1$-dimensional space form $\mathbb
E^{n+1}(c)$ is biharmonic if and only if
\begin{equation}
\begin{cases}
\Delta H+H {\rm trace}\, A^2 =ncH,\\
2A\,{\rm grad}H+n\, H{\rm grad}H=0.
\end{cases}
\end{equation}
\end{theorem}
Recall a result on biharmonic hypersurfaces with at most three
distinct principal curvatures in $\mathbb E^{n+1}(c)$ in \cite{fuyu}
for later use.
\begin{theorem}
Let $M^n$ be a proper biharmonic hypersurface with at most three
distinct principal curvatures in $\mathbb E^{n+1}(c)$. Then $M^n$
has constant mean curvature.
\end{theorem}
\section{Biharmonic hypersurfaces with constant Gauss scalar curvature in $\mathbb E^{5}(c)$}
We restrict ourselves to biharmonic hypersurfaces $M$ in the
5-dimensional space form $\mathbb E^5(c)$.

Assume that the mean curvature $H$ is not constant.

It follows from the second equation of (2.6) that ${\rm grad}\,H$ is
an eigenvector of the shape operator $A$ with the corresponding
principal curvature $-2H$. Therefore, without loss of generality, we
choose $e_1$ such that $e_1$ is parallel to ${\rm grad}\,H$, and
hence the shape operator $A$ of $M$ takes the following form with
respect to some suitable orthonormal frame $\{e_1, e_2, e_3, e_4\}$
\begin{eqnarray}
Ae_i=\lambda_i e_i,
\end{eqnarray}
where $\lambda_1=-2H$.

Denote by $R$ the scalar curvature and by $B$ the squared length of
the second fundamental form $h$ of $M$. It follows from (3.1) that
$B$ is given by
\begin{eqnarray}
B=\sum_{i=1}^4\lambda^2_i=4H^2+\lambda_2^2+\lambda_3^2+\lambda_4^2.
\end{eqnarray}
From the Gauss equation, the scalar curvature $R$ is given by
\begin{eqnarray}
R=12c+16H^2-B=12c+12H^2-\lambda_2^2-\lambda_3^2-\lambda_4^2.
\end{eqnarray}
We compute ${\rm grad}\,H$ as
\begin{eqnarray*}
{\rm grad}\,H=\sum_{i=1}^4e_i(H)e_i.
\end{eqnarray*}
Since $e_1$ is parallel to ${\rm grad}\,H$, it follows that
\begin{eqnarray}
e_1(H)\neq0,\quad e_2(H)=e_3(H)=e_4(H)=0.
\end{eqnarray}
We set
\begin{eqnarray}
\nabla_{e_i}e_j=\sum_{k=1}^4\omega_{ij}^ke_k,\quad i,j=1, 2, 3, 4.
\end{eqnarray}
The compatibility conditions $\nabla_{e_k}\langle e_i,e_i\rangle=0$
and $\nabla_{e_k}\langle e_i,e_j\rangle=0$ $(i\neq j)$ give,
respectively, that
\begin{eqnarray}
\omega_{ki}^i=0,\quad \omega_{ki}^j+\omega_{kj}^i=0,
\end{eqnarray}
for $i\neq j$ and $i, j, k=1, 2, 3, 4$. From (3.1) and (3.4), the
Codazzi equation leads to
\begin{eqnarray}
e_i(\lambda_j)=(\lambda_i-\lambda_j)\omega_{ji}^j,\\
(\lambda_i-\lambda_j)\omega_{ki}^j=(\lambda_k-\lambda_j)\omega_{ik}^j
\end{eqnarray}
for distinct $i, j, k=1, 2, 3, 4$.

Since $\lambda_1=-2H$, from (3.4) we compute that
\begin{eqnarray*}
[e_2,e_3](\lambda_1)=[e_3,e_4](\lambda_1)=[e_2,e_4](\lambda_1)=0,
\end{eqnarray*}
which yields directly
\begin{eqnarray}
\omega_{ij}^1=\omega_{ji}^1, \quad i, j=2, 3, 4 ~\,{\rm and}\,
~i\neq j.
\end{eqnarray}
Now we claim that $\lambda_j\neq\lambda_1$ for $j=2, 3, 4$. In fact,
if $\lambda_j=\lambda_1$ for $j\neq1$, by putting $i=1$ in (3.7) we
have that
\begin{eqnarray}
0=(\lambda_1-\lambda_j)\omega_{j1}^j=e_1(\lambda_j)=e_1(\lambda_1).
\end{eqnarray}
However, (3.10) contradicts to the first expression of (3.4).

According to Theorem 2.1, we only need to deal with the case for $M$
to have four distinct principal curvatures. Hence, we assume that
$M$ has our distinct principal curvatures in the following.

By the definition (2.4) of the mean curvature vector field
$\overrightarrow{H}$ and $\lambda_1=-2H$, we have
\begin{eqnarray}
\lambda_2+\lambda_3+\lambda_4=6H
\end{eqnarray}
for distinct $\lambda_2, \lambda_3, \lambda_4$ and $\lambda_i\neq
-2H$.

We now state a lemma to express the connection coefficients of $M$.
\begin{lemma}
Let $M$ be a biharmonic hypersurface with four distinct principal
curvatures in space forms $\mathbb E^5(c)$, whose shape operator
given by (3.1) with respect to an orthonormal frame $\{e_1, e_2,
e_3, e_4\}$. Then we have
\begin{eqnarray*}
&&\nabla_{e_1}e_i=0,\quad i=1, 2, 3, 4,\\
&&\nabla_{e_i}e_1=-\omega_{ii}^1e_i,\quad i=2, 3, 4,\\
&&\nabla_{e_i}e_i=\sum_{k=1, k\neq i}^{4}{\omega_{ii}^k}e_k,\quad i=2, 3, 4,\\
&&\nabla_{e_i}e_j=-{\omega_{ii}^j}e_i+\omega_{ij}^ke_k ~for~
distinct ~i, j, k=2, 3, 4,
\end{eqnarray*}
where
\begin{eqnarray*}
\omega_{ii}^j=-\frac{e_j(\lambda_i)}{\lambda_j-\lambda_i}.
\end{eqnarray*}
\end{lemma}
\begin{proof}
Consider the equations (3.7) and (3.8).

By putting $j=1$ and $i=2, 3, 4$ in (3.7), from (3.4) we have
$\omega_{1i}^1=0$, which together with the first expression of (3.6)
gives
\begin{eqnarray}
\omega_{1i}^1=0,\quad i=1, 2, 3, 4.
\end{eqnarray}
Combining (3.12) with the second expression of (3.6) gives
\begin{eqnarray}
\omega_{11}^i=0,\quad i=1, 2, 3, 4.
\end{eqnarray}

By putting $j=1$, $i, k=2, 3, 4$ in (3.8), and applying (3.9) we
have
\begin{eqnarray}
\omega_{ij}^1=\omega_{ji}^1=0,
\end{eqnarray}
which together with the second expression of (3.6) yields
\begin{eqnarray}
\omega_{i1}^j=0,\quad i, j=2, 3, 4,\, {\rm and}~\, i\neq j.
\end{eqnarray}
By applying (3.8) again, from (3.15) it follows that
\begin{eqnarray}
\omega_{1i}^j=0,\quad i, j=2, 3, 4,\, {\rm and}~\, i\neq j.
\end{eqnarray}
Combining (3.12-3.16) with (3.6) and (3.7), we complete the proof of
Lemma 3.1.
\end{proof}
Since the Gauss curvature tensor $R(X,Y)Z$ is defined by
\begin{eqnarray*}
R(X,Y)Z=\nabla_X\nabla_YZ-\nabla_Y\nabla_XZ-\nabla_{[X,Y]}Z,
\end{eqnarray*}
we could compute the curvature tensor $R$ by Lemma 3.1. On the other
hand, by applying the Gauss equation for different values of $X$,
$Y$ and $Z$ and by comparing the coefficients with respect to the
orthonormal basis $\{e_1, e_2, e_3, e_4\}$ we get the following:
\begin{itemize}
\item $X=e_1, Y=e_2, Z=e_1$,
\begin{eqnarray}
e_1(\omega_{22}^1)-(\omega_{22}^1)^2=\lambda_1\lambda_2+c;
\end{eqnarray}
\item $X=e_1, Y=e_3, Z=e_1$,
\begin{eqnarray}
e_1(\omega_{33}^1)-(\omega_{33}^1)^2=\lambda_1\lambda_3+c;
\end{eqnarray}
\item $X=e_1, Y=e_4, Z=e_1$,
\begin{eqnarray}
e_1(\omega_{44}^1)-(\omega_{44}^1)^2=\lambda_1\lambda_4+c;
\end{eqnarray}
\item $X=e_1, Y=e_3, Z=e_3$,
\begin{eqnarray}
e_1(\omega_{33}^2)=\omega_{33}^1\omega_{33}^2;
\end{eqnarray}
\item $X=e_1, Y=e_4, Z=e_4$,
\begin{eqnarray}
e_1(\omega_{44}^2)=\omega_{44}^1\omega_{44}^2;
\end{eqnarray}
\item $X=e_2, Y=e_3, Z=e_3$,
\begin{eqnarray}
e_2(\omega_{33}^1)=-\omega_{22}^1\omega_{33}^2+\omega_{33}^1\omega_{33}^2;
\end{eqnarray}
\item $X=e_2, Y=e_4, Z=e_4$,
\begin{eqnarray}
e_2(\omega_{44}^1)=-\omega_{22}^1\omega_{44}^2+\omega_{44}^1\omega_{44}^2;
\end{eqnarray}
\item $X=e_2, Y=e_3, Z=e_2$,
\begin{eqnarray}
&&-e_2(\omega_{33}^2)-e_3(\omega_{22}^3)+\omega_{22}^4\omega_{33}^4+(\omega_{22}^3)^2+(\omega_{33}^2)^2\nonumber\\
&&+\omega_{22}^1\omega_{33}^1-\omega_{24}^3\omega_{34}^2-\omega_{24}^3\omega_{43}^2+\omega_{34}^2\omega_{43}^2=-(c+\lambda_2\lambda_3);
\end{eqnarray}
\item $X=e_2, Y=e_4, Z=e_2$,
\begin{eqnarray}
&&-e_4(\omega_{22}^4)-e_2(\omega_{44}^2)+\omega_{22}^3\omega_{44}^3+(\omega_{22}^4)^2+(\omega_{44}^2)^2\nonumber\\
&&+\omega_{22}^1\omega_{44}^1+\omega_{24}^3\omega_{34}^2+\omega_{24}^3\omega_{43}^2+\omega_{34}^2\omega_{43}^2=-(c+\lambda_2\lambda_4);
\end{eqnarray}
\item $X=e_3, Y=e_4, Z=e_3$,
\begin{eqnarray}
&&-e_3(\omega_{44}^3)-e_4(\omega_{33}^4)+\omega_{33}^2\omega_{44}^2+(\omega_{33}^4)^2+(\omega_{44}^3)^2\nonumber\\
&&+\omega_{33}^1\omega_{44}^1+\omega_{24}^3\omega_{34}^2-\omega_{24}^3\omega_{43}^2-\omega_{34}^2\omega_{43}^2=-(c+\lambda_3\lambda_4).
\end{eqnarray}
\end{itemize}
Note that in the above we only state the equations useful for later
use.

 Consider the first equation of (2.6). It follows from
(3.1), (3.3), and Lemma 3.1 that
\begin{equation}
-e_1e_1(H)+(\omega_{22}^1+\omega_{33}^1+\omega_{44}^1)e_1(H)+H(8c+16H^2-R)=0.
\end{equation}
Let us compute $[e_1,e_i](H)=(\nabla_{e_1}e_i-\nabla_{e_i}e_1)(H)$
for $i=2, 3, 4$. From (3.4) and Lemma 3.1, it follows that
\begin{eqnarray}
e_ie_1(H)=0,\quad i=2, 3, 4.
\end{eqnarray}
\begin{lemma}
Let $M$ be a biharmonic hypersurface with four distinct principal
curvatures in space forms $\mathbb E^5(c)$, and whose shape operator
given by (3.1) with respect to an orthonormal frame $\{e_1, e_2,
e_3, e_4\}$. If the scalar curvature $R$ is constant, then
$e_i(\lambda_j)=0$ for $i=2, 3, 4$ and $j=1, 2, 3, 4$.
\end{lemma}
\begin{proof}
By the hypothesis, the scalar curvature $R$ is constant.
Differentiating (3.27) along $e_2$, from (3.28) we have
\begin{eqnarray}
e_2(\omega_{22}^1+\omega_{33}^1+\omega_{44}^1)=0.
\end{eqnarray}
On the other hand, differentiating (3.11) along $e_1$, by (3.7) and
the second equation of (3.6) we obtain
\begin{eqnarray}
(\lambda_1-\lambda_2)\omega_{22}^1+(\lambda_1-\lambda_3)\omega_{33}^1+(\lambda_1-\lambda_4)\omega_{44}^1=-6e_1(H),
\end{eqnarray}
which reduces to
\begin{eqnarray*}
\omega_{22}^1+\omega_{33}^1+\omega_{44}^1=-\frac{\lambda_2-\lambda_3}{\lambda_1-\lambda_2}\omega_{33}^1
-\frac{\lambda_2-\lambda_4}{\lambda_1-\lambda_2}\omega_{44}^1-\frac{6e_1(H)}{\lambda_1-\lambda_2}.
\end{eqnarray*}
Now acting $e_2$ on both sides of the above equation, by (3.11),
(3.4), (3.22) and (3.23) we obtain
\begin{eqnarray}
e_2(\omega_{22}^1+\omega_{33}^1+\omega_{44}^1)&=&\Big(\frac{2e_2(\lambda_3)+e_2(\lambda_4)}{\lambda_1-\lambda_2}+
\frac{(\lambda_2-\lambda_3)\big(e_2(\lambda_3)+e_2(\lambda_4)\big)}{(\lambda_1-\lambda_2)^2}\Big)\omega_{33}^1\nonumber\\
&+&\Big(\frac{e_2(\lambda_3)+2e_2(\lambda_4)}{\lambda_1-\lambda_2}+
\frac{(\lambda_2-\lambda_4)\big(e_2(\lambda_3)+e_2(\lambda_4)\big)}{(\lambda_1-\lambda_2)^2}\Big)\omega_{44}^1\nonumber\\
&+&\frac{6e_1(H)\big(e_2(\lambda_3)+e_2(\lambda_4)\big)}{(\lambda_1-\lambda_2)^2}+
\frac{\lambda_2-\lambda_3}{\lambda_1-\lambda_2}\omega_{33}^2(\omega_{22}^1-\omega_{33}^1)+\nonumber\\
&+&\frac{\lambda_2-\lambda_4}{\lambda_1-\lambda_2}\omega_{44}^2(\omega_{22}^1-\omega_{44}^1)\nonumber\\
&=&\frac{(2\lambda_1-\lambda_2-\lambda_3)e_2(\lambda_3)+(\lambda_1-\lambda_3)e_2(\lambda_4)}{(\lambda_1-\lambda_2)^2}\omega_{33}^1\nonumber\\
&+&\frac{(\lambda_1-\lambda_4)e_2(\lambda_3)+(2\lambda_1-\lambda_2-\lambda_4)e_2(\lambda_4)}{(\lambda_1-\lambda_2)^2}\omega_{44}^1\nonumber\\
&-&\frac{(\lambda_1-\lambda_2)\omega_{22}^1+(\lambda_1-\lambda_3)\omega_{33}^1+(\lambda_1-\lambda_4)\omega_{44}^1}{(\lambda_1-\lambda_2)^2}
\big(e_2(\lambda_3)+e_2(\lambda_4)\big)\nonumber\\
&+&\frac{\lambda_2-\lambda_3}{\lambda_1-\lambda_2}\omega_{33}^2(\omega_{22}^1-\omega_{33}^1)+
\frac{\lambda_2-\lambda_4}{\lambda_1-\lambda_2}\omega_{44}^2(\omega_{22}^1-\omega_{44}^1).
\end{eqnarray}
From (3.7) and the second expression of (3.6), we have
$e_2(\lambda_3)=-(\lambda_2-\lambda_3)\omega_{33}^2$ and
$e_2(\lambda_4)=-(\lambda_2-\lambda_4)\omega_{44}^2$. Substituting
these into (3.31) gives
\begin{equation}
e_2(\omega_{22}^1+\omega_{33}^1+\omega_{44}^1)=\frac{2(\lambda_2-\lambda_3)}{\lambda_1-\lambda_2}
(\omega_{22}^1-\omega_{33}^1)\omega_{33}^2+
\frac{2(\lambda_2-\lambda_4)}{\lambda_1-\lambda_2}
(\omega_{22}^1-\omega_{44}^1)\omega_{44}^2.
\end{equation}
Combining (3.32) with (3.29) gives
\begin{eqnarray}
(\lambda_2-\lambda_3)(\omega_{22}^1-\omega_{33}^1)\omega_{33}^2
+(\lambda_2-\lambda_4)(\omega_{22}^1-\omega_{44}^1)\omega_{44}^2=0.
\end{eqnarray}
Moreover, differentiating (3.3) along $e_2$, by (3.11) and (3.7) we
have
\begin{eqnarray}
(\lambda_2-\lambda_3)^2\omega_{33}^2+(\lambda_2-\lambda_4)^2\omega_{44}^2=0.
\end{eqnarray}
Differentiating (3.34) along $e_1$, by applying (3.7), the second
expression of (3.6), (3.20) and (3.21) we obtain
\begin{eqnarray}
&&(\lambda_2-\lambda_3)\Big[2(\lambda_1-\lambda_2)\omega_{22}^1-(2\lambda_1+\lambda_2-3\lambda_3)\omega_{33}^1
\Big]\omega_{33}^2\\
&&+(\lambda_2-\lambda_4)\Big[2(\lambda_1-\lambda_2)\omega_{22}^1-(2\lambda_1+\lambda_2-3\lambda_4)\omega_{44}^1
\Big]\omega_{44}^2=0.\nonumber
\end{eqnarray}
We claim that $\omega_{33}^2=\omega_{44}^2=0$.

In fact, if one of $\omega_{33}^2$ and $\omega_{44}^2$ is not
vanishing, (3.33) and (3.34) imply that
\begin{eqnarray}
(\lambda_3-\lambda_4)\omega_{22}^1-(\lambda_2-\lambda_4)\omega_{33}^1
+(\lambda_2-\lambda_3)\omega_{44}^1=0.
\end{eqnarray}
Also, (3.34) and (3.35) reduce to
\begin{eqnarray}
2(\lambda_1-\lambda_2)(\lambda_3-\lambda_4)\omega_{22}^1-(\lambda_2-\lambda_4)(2\lambda_1+\lambda_2-3\lambda_3)\omega_{33}^1\\
+(\lambda_2-\lambda_3)(2\lambda_1+\lambda_2-3\lambda_4)\omega_{44}^1=0.\nonumber
\end{eqnarray}
Eliminating $\omega_{22}^1$ between (3.36) and (3.37) gives
\begin{eqnarray*}
3(\lambda_2-\lambda_3)(\lambda_2-\lambda_4)(\omega_{33}^1-\omega_{44}^1)=0,
\end{eqnarray*}
which yields
\begin{eqnarray}
\omega_{33}^1=\omega_{44}^1.
\end{eqnarray}
Substituting (3.38) into (3.36), we obtain
\begin{eqnarray}
\omega_{22}^1=\omega_{33}^1.
\end{eqnarray}
Acting $e_1$ on both sides of (3.3) and (3.11), by using (3.7) and
the second expression of (3.6) we obtain a relation
\begin{eqnarray}
(\lambda_1-\lambda_2)(2\lambda_2-\lambda_3-\lambda_4)\omega_{22}^1-(\lambda_1-\lambda_3)(\lambda_2-2\lambda_3+\lambda_4)\omega_{33}^1\\
-(\lambda_1-\lambda_4)(\lambda_2+\lambda_3-2\lambda_4)\omega_{44}^1=0,\nonumber
\end{eqnarray}
which together with (3.38) and (3.39) yields
\begin{eqnarray}
\big[(\lambda_2-\lambda_3)^2+(\lambda_2-\lambda_4)^2+(\lambda_3-\lambda_4)^2\big]\omega_{22}^1=0.
\end{eqnarray}
Since the principal curvatures $\lambda_i$ $(i=1, 2, 3, 4)$ are
mutually different, it follows from (3.41), (3.38) and (3.39) that
\begin{eqnarray}
\omega_{22}^1=\omega_{33}^1=\omega_{44}^1=0.
\end{eqnarray}
Combining (3.30) with (3.42) gives $e_1(H)=0$, which contradicts to
the first expression of (3.4).

Therefore, we conclude $\omega_{33}^2=\omega_{44}^2=0$. By (3.7),
(3.4) and (3.11), we obtain $e_2(\lambda_i)=0$ for $i=1, 2, 3, 4$.

With some similar discussions, we could show that
$e_3(\lambda_i)=e_4(\lambda_i)=0$ for $i=1, 2, 3, 4$ as well. This
completes the proof of Lemma 3.2.
\end{proof}
We are ready to state the main theorem.
\begin{theorem}
Let $M$ be a proper biharmonic hypersurface with constant scalar
curvature in $\mathbb E^{5}(c)$. Then $M$ has constant mean
curvature.
\end{theorem}
\begin{proof}
By Lemma 3.2, equations (3.24-3.26), respectively, reduce to
\begin{eqnarray}
&&\omega_{22}^1\omega_{33}^1-\omega_{24}^3\omega_{34}^2-\omega_{24}^3\omega_{43}^2+\omega_{34}^2\omega_{43}^2=-(c+\lambda_2\lambda_3),\\
&&\omega_{22}^1\omega_{44}^1+\omega_{24}^3\omega_{34}^2+\omega_{24}^3\omega_{43}^2+\omega_{34}^2\omega_{43}^2=-(c+\lambda_2\lambda_4),\\
&&\omega_{33}^1\omega_{44}^1+\omega_{24}^3\omega_{34}^2-\omega_{24}^3\omega_{43}^2-\omega_{34}^2\omega_{43}^2=-(c+\lambda_3\lambda_4).
\end{eqnarray}
Moreover, it follows from (3.8) and the second expression of (3.6)
that
\begin{eqnarray}
\omega_{24}^3\omega_{34}^2=\omega_{24}^3\omega_{43}^2-\omega_{34}^2\omega_{43}^2,\\
(\lambda_3-\lambda_4)\omega_{24}^3=(\lambda_2-\lambda_4)\omega_{34}^2.
\end{eqnarray}
Eliminating $\omega_{24}^3$, $\omega_{34}^2$ and $\omega_{43}^2$
from (3.43-3.45) by using (3.46), (3.47), (3.11) and (3.3), we
obtain
\begin{eqnarray}
&&\omega_{22}^1\omega_{33}^1+\omega_{22}^1\omega_{44}^1+\omega_{33}^1\omega_{44}^1=-12H^2+3c-\frac{1}{2}R,\\
&&\lambda_3\omega_{22}^1\omega_{44}^1+\lambda_2\omega_{33}^1\omega_{44}^1+\lambda_4\omega_{22}^1\omega_{33}^1=-6cH-3\lambda_2\lambda_3\lambda_4.
\end{eqnarray}
By (3.7) and the second expression of (3.6), we rewrite (3.17-3.19),
respectively, as follows:
\begin{eqnarray}
\quad e_1e_1(\lambda_2)+\omega_{22}^1e_1(\lambda_1)+2(\lambda_1-\lambda_2)(\omega_{22}^1)^2+(\lambda_1-\lambda_2)(\lambda_1\lambda_2+c)=0,\\
\quad e_1e_1(\lambda_3)+\omega_{33}^1e_1(\lambda_1)+2(\lambda_1-\lambda_3)(\omega_{33}^1)^2+(\lambda_1-\lambda_3)(\lambda_1\lambda_3+c)=0,\\
\quad
 e_1e_1(\lambda_4)+\omega_{44}^1e_1(\lambda_1)+2(\lambda_1-\lambda_4)(\omega_{44}^1)^2+(\lambda_1-\lambda_4)(\lambda_1\lambda_4+c)=0.
\end{eqnarray}
Since $\lambda_1=-2H$, eliminating $e_1e_1(H)$ from (3.27) and
(3.50-3.52), by (3.30), (3.11) and (3.3) we have
\begin{eqnarray}
4(\omega_{22}^1+\omega_{33}^1+\omega_{44}^1)e_1(H)+48H^3-66cH+9RH-3\lambda_2\lambda_3\lambda_4=0.
\end{eqnarray}
It follows from (3.53) that (3.27) reduces to
\begin{eqnarray}
4e_1e_1(H)-16H^3-98cH+13RH-3\lambda_2\lambda_3\lambda_4=0.
\end{eqnarray}
Now, by the fact $\lambda_1=-2H$ and (3.11), (3.40) becomes
\begin{eqnarray}
(\lambda_2^2-4H^2)\omega_{22}^1+(\lambda_3^2-4H^2)\omega_{33}^1+(\lambda_4^2-4H^2)\omega_{44}^1=0.
\end{eqnarray}
From (3.3) and (3.11), we have that
\begin{eqnarray}
\lambda_3\lambda_4=\frac{1}{2}R-6c+12H^2-6H\lambda_2+\lambda_2^2,\\
\lambda_2\lambda_4=\frac{1}{2}R-6c+12H^2-6H\lambda_3+\lambda_3^2,\\
\lambda_3\lambda_4=\frac{1}{2}R-6c+12H^2-6H\lambda_4+\lambda_4^2.
\end{eqnarray}
Hence, from (3.55-3.58), (3.7), (3.30) we get
\begin{eqnarray}
e_1(\lambda_2\lambda_3\lambda_4)=-(\lambda_1-\lambda_2)\lambda_3\lambda_4\omega_{22}^1
-(\lambda_1-\lambda_3)\lambda_2\lambda_4\omega_{33}^1-(\lambda_1-\lambda_4)\lambda_2\lambda_3\omega_{44}^1\nonumber\\
=(56H^3+RH-12cH+\lambda_2\lambda_3\lambda_4)(\omega_{22}^1+\omega_{33}^1+\omega_{44}^1)-72H^2e_1(H).
\end{eqnarray}

Differentiating (3.53) along $e_1$, by using (3.17-3.19), (3.54),
(3.53) and (3.59) we obtain that
\begin{eqnarray}
(200H^3+25RH-200cH-3\lambda_2\lambda_3\lambda_4)(\omega_{22}^1+\omega_{33}^1+\omega_{44}^1)\\
=(160H^2+13R-78c)e_1(H).\nonumber
\end{eqnarray}
Combining (3.60) with (3.53) gives
\begin{eqnarray}
4\big(e_1(H)\big)^2(160H^2+13R-78c) =-(48H^3-66cH+9RH
\\-3\lambda_2\lambda_3\lambda_4)(200H^3+25RH-200cH-3\lambda_2\lambda_3\lambda_4).\nonumber
\end{eqnarray}
Now differentiating (3.61) along $e_1$, using (3.54), (3.59),
(3.60), (3.61) we have an algebraic equation concerning $H$ and
$\lambda_2\lambda_3\lambda_4$ with constant coefficients
\begin{eqnarray}
&&2040217600H^{10}+(659304960R-4882549760c)H^8\\
&&+(3730891264c^2-1021023488cR+69428224R^2)H^6\nonumber\\
&&+(-987669696c^3-55470688cR^2+407658368c^2R+2493816R^3)H^4\nonumber\\
&&+(115086816c^4-55092024c^3R+9593272c^2R^2
-716326cR^3 +19162R^4)H^2\nonumber\\
&&-74403840H^7K+(105242112c-15432192R)H^5K\nonumber\\
&&+(-927984R^2
+12200976cR-38310432c^2)H^3K\nonumber\\
&&+(11289096c^3-4544436c^2R +602004cR^2-26364R^3)HK\nonumber\\
&&+403200H^4K^2 +(133488c+16632R)H^2K^2+8640HK^3\nonumber\\
&&+(186732c^2-54990Rc +3978R^2)K^2=0,\nonumber
\end{eqnarray}
where $K=\lambda_2\lambda_3\lambda_4$.

If $K$ is a constant, then (3.62) reduces to an algebraic equation
of $H$ with constant coefficients. Thus, the real function $H$
satisfies a polynomial equation $q(H)=0$ with constant coefficients,
therefore it must be a constant. We get a contradiction.

Assume that $K$ is not constant. Consider an integral curve of $e_1$
passing through $p=\gamma(t_0)$ as $\gamma(t), t\in I$. According to
Lemma 3.2, we can assume $t=t(K)$ and $H=H(K)$ in some neighborhood
of $K_0=K(t_0)$.

Note that
\begin{eqnarray}
\frac{dH}{d K}=\frac{dH}{dt}\frac{dt}{dK}=\frac{e_1(H)}{e_1(K)}.
\end{eqnarray}
In fact, equations (3.59) and (3.60) could yield
\begin{eqnarray}
\quad
\frac{e_1(K)}{e_1(H)}=\frac{(56H^3+RH-12cH+K)(160H^2+13R-78c)}{(200H^3+25RH-200cH-3K)}-72H^2.
\end{eqnarray}

Differentiating (3.62) with respect to $K$ and substituting
$\frac{dH}{dK}$ from (3.63) and (3.64), we get another independent
algebraic equation of  $H$ and $K$
\begin{eqnarray}
\sum_{i=0}^4q_i(H)K^i=0.
\end{eqnarray}
where $q_i(H)$ is a polynomial concerning function $H$.

We may eliminate $K^4$, $K^3$, $K^2$ and $K$ from equations (3.62)
and (3.65) gradually. At last, we obtain a non-trivial algebraic
polynomial equation of $H$ with constant coefficients. Therefore, we
conclude that the real function $H$ must be a constant, which
contradicts to our original assumption. This completes the proof of
Theorem 3.3.
\end{proof}
As a corollary, we immediately get the following characterization
theorem.
\begin{theorem}
Every biharmonic hypersurface with constant scalar curvature in the
5-dimensional sphere $\mathbb S^5$ has constant mean curvature.
\end{theorem}
\begin{remark}
Since all the known examples of proper biharmonic submanifolds in
$\mathbb S^n$ have constant mean curvature, Balmus-Montaldo-Oniciuc
in \cite{BMO2008} conjectured that {\em the proper biharmonic
hypersurfaces in $\mathbb S^{n+1}$ must have constant mean
curvature}. Hence, Theorem 3.4 gives an affirmative partial answer
to this conjecture.
\end{remark}
Consider the cases $c=0, -1$ in the calculation above. From the
first equation of (2.6), one can easily obtain
\begin{theorem}
There exist no proper biharmonic hypersurfaces with constant scalar
curvature in the 5-dimensional Euclidean space $\mathbb E^5$ or
hyperbolic space $\mathbb H^5$.
\end{theorem}
\begin{remark}
Theorem 3.6 gives affirmative partial answers to Chen's conjecture
and Generalized Chen's conjecture.
\end{remark}
We end this paper with a further remark.
\begin{remark}
Replace the condition {\em constant scalar curvature} by {\em
constant length of the second fundamental form} in Theorem 3.3. In
view of expressions (3.2) and (3.3), with quite similar argument as
above we could obtain similar conclusions.
\end{remark}


\end{document}